\documentclass{amsart}
\usepackage{amsmath}
\usepackage{amssymb}
\usepackage{amsthm}
\usepackage{bm}
\usepackage[pdftex]{graphicx}
\usepackage{tikz}

\newtheorem{theorem}{Theorem}
\newtheorem{prop}{Proposition}
\newtheorem{lemma}{Lemma}
\newtheorem{rem}{Remark}
\newtheorem{exmp}{Example}

\begin{document}

\title{Graphs related to $2$-dimensional simplex codes}
\author{Mariusz Kwiatkowski, Mark Pankov}

\address{Faculty of Mathematics and Computer Science, 
University of Warmia and Mazury, S{\l}oneczna 54, Olsztyn, Poland}
\email{mkw@matman.uwm.edu.pl, pankov@matman.uwm.edu.pl}

\maketitle
\begin{abstract}
We give a complete description of the distance relation on the graph of $4$-ary simplex codes of dimension $2$.
This is a connected graph of diameter $3$.
For every vertex we determine the sets of all vertices at distance $i\in\{1,2,3\}$ and describe their symmetries.
\end{abstract}

\maketitle
\section{Introduction}
The distance between two vertices in a connected graph is equal to the number of edges in a shortest path connecting these vertices. 
For a connected subgraph the distance between two vertices (in this subgraph)  is equal to or greater than the distance between these vertices in the graph.
We consider the Grassmann graph consisting of all linear $[n,k]_q$ codes,
i.e. $k$-dimensional subspaces of an $n$-dimensional vector space over the $q$-element field, 
and  its subgraphs formed by codes of special type (non-degenerate, projective, simplex). 
Two distinct codes are adjacent vertices of the Grassmann graph if they have the maximal possible number of common codewords,
i.e. the dimension of their intersection is equal $k-1$.
The case when $k=1,n-1$ is trivial (any two distinct vertices of the Grassmann graph are adjacent) 
and we will always suppose that $1<k<n-1$.

The subgraph of non-degenerated linear $[n,k]_q$ codes is investigated in \cite{KP}.
It is not difficult to prove that this subgraph is connected and the distance between any two vertices is not greater than $i+1$,
where $i$ is the distance between these vertices in the Grassmann graph. 
The main result of \cite{KP} states that the distance  in the subgraph (for any pair of vertices)
coincides  with the distance in the Grassmann graph if and only if 
$$ n<(q+1)^2 +k-2.$$
The subgraph of projective codes is considered in  \cite{KPP}.
If the number of elements in the field is sufficiently large, i.e. $q\ge \binom{n}{2}$, then 
this graph is connected and all distances coincide with the distances in the Grassmann graph.
Projective $[n,k]_q$ codes exist if and only if
$$\frac{q^k-1}{q-1}\ge n;$$
in the case when this is an equality, projective codes are called {\it $q$-ary simplex codes of dimension} $k$. 

Simplex codes are interesting for many reasons.
They are dual to well-known Hamming codes and are first-order Reed--Muller codes \cite[Subsection 1.2.2]{TVN}.
The associated projective systems are projective spaces which guarantees that
all simplex codes with the same parameters are equivalent, i.e. the group of monomial linear automorphism acts transitively on the set of such codes.

The inequality $q\ge \binom{n}{2}$ does not hold for $q$-ary simplex codes of dimension $k$
and there are pairs of such codes for which  the distance  in the subgraph of simplex codes  is greater than 
the corresponding distance in the Grassmann graph (this happens, for example, if $k=4, q=2$ or $k=3, q=3$, see \cite{KPP}).
On the other hand, for $k=3,q=2$ and $k=2,q=3$ the subgraph of simplex codes is isomorphic to the dual polar graph ${\textsf D}_{k}(q)$ 
(see \cite{KPP} and Example \ref{exmp-grid}, respectively).
This  follows easily from the fact that simplex codes can be characterized as maximal linear subspaces contained in a certain algebraic variety.
In the general case, the corresponding system of polynomial equations is complicated 
(for example, the number of equation is $mk-m$ if $q=p^m$ and $p$ is a prime number).
This makes the general treatment of the subgraph of simplex codes highly  problematic. 

In the present paper, we focus on the subgraph of $q$-ary simplex codes of dimension $2$.
The case when $q=3$ is simple (it was noted above). 
For $q=4$ we give a complete description of the distance relation on this subgraph  (Theorem \ref{theorem1}).
This is a connected graph of diameter $3$ (note that the diameter of the corresponding Grassmann graph is $2$).
For every vertex we determine the sets of all vertices at distance $i\in\{1,2,3\}$.  
The second result (Theorem \ref{theorem2}) concerns the symmetries of these sets.
By duality, we can obtain the same description for the subgraph of $4$-ary Hamming codes of dimension $3$.

Our arguments are based on the following interesting properties of the field with four elements:
\begin{enumerate}
\item[$\bullet$] For the $q$-element field with $q>2$ the sum of all $q-1$ non-zero elements is zero,
but only  for $q=3,4$ the sum of $q-1$ non-zero elements is zero {\it if and only if} these elements are mutually distinct.
\item[$\bullet$] The group ${\rm P\Gamma L}(2,4)$ is isomorphic to the permutation group ${\rm S}_5$.
\end{enumerate}
We are not able to extend the methods on the case when $q>4$ 
(see Proposition \ref{prop-ad} and Remark \ref{rem-ad} for the explanation).

\section{Basic objects}
\subsection{Grassmann graph}
Consider an $n$-dimensional vector space $V={\mathbb F}^{n}$
over the finite field ${\mathbb F}=\mathbb{F}_{q}$ consisting of $q$ elements. 
The standard basis of $V$ is formed by the vectors 
$$e_{1}=(1,0,\dots,0),\dots,e_{n}=(0,\dots,0,1).$$
We write $c_{i}$ for the $i$-th coordinate functional $(x_{1},\dots,x_{n})\to x_{i}$
and denote by $C_{i}$ the hyperplane of $V$ which is the kernel of $c_i$.
Recall that an $m$-dimensional vector space over ${\mathbb F}_q$ contains precisely 
$$[m]_{q}=\frac{q^{m}-1}{q-1}=q^{m-1}+\dots+q+1$$
distinct $1$-dimensional subspaces and this number coincides with the number of hyperplanes in this vector space.

The {\it Grassmann graph} $\Gamma_{k}(V)$ is the simple graph whose vertices are all $k$-dimen\-sional subspaces of $V$. 
Two $k$-dimensional  subspaces are adjacent vertices of this graph (connected by an edge)
if their intersection is $(k-1)$-dimensional.
In the case when $k=1,n-1$, any two distinct vertices of $\Gamma_{k}(V)$ are adjacent. 
In what follows, we will always assume that $1<k<n-1$. 

For a $(k-1)$-dimensional subspace $X\subset V$ and a $(k+1)$-dimensional subspace $Y\subset V$
the {\it star} ${\mathcal S}(X)$ and the {\it top} ${\mathcal T}(Y)$ consist of all $k$-dimensional subspaces 
containing $X$ and contained in $Y$, respectively. 
These are cliques in $\Gamma_{k}(V)$  (i.e. any two distinct elements are adjacent vertices).
Conversely, every clique of $\Gamma_{k}(V)$ is contained in a star or a top 
which means that stars and tops are maximal cliques of the Grassmann graph \cite[Proposition 3.3]{Pank-book}.

The distance between two $k$-dimensional subspaces of $V$ in $\Gamma_{k}(V)$ is equal to $k-\dim(X\cap Y)$.

\subsection{Simplex codes}
Let $C$ be a linear $[n,k]_{q}$ code, i.e. a $k$-dimensional subspace of $V$.
We assume that this code is non-degenerate which means that the restriction of every coordinate functional to $C$ is non-zero.
Suppose that $x_{1},\dots,x_{k}$ is a basis of $C$ and $x^{*}_{1},\dots,x^{*}_{k}$ is the dual basis of $C^{*}$
satisfying $x^{*}_{i}(x_{j})=\delta_{ij}$ (Kronecker delta).
If $M$ is the generator matrix of $C$ whose rows are $x_{1},\dots,x_{k}$ and
$(a_{1j},\dots, a_{kj})$ is the $j$-column of $M$, then 
$$c_{j}|_{C}= a_{1j}x^{*}_{1}+\dots+a_{kj}x^{*}_{k}.$$
Denote by $P_{i}$ the $1$-dimensional subspace of the dual vector space $C^{*}$ containing $c_{i}|_{C}$.
The collection ${\mathcal P}(C)=\{P_{1},\dots,P_{n}\}$ is known as the {\it projective system} associated to $C$ \cite{TVN}.
In the general case, the equality $P_{i}=P_{j}$ is possible for some distinct $i,j$.
Our next assumption is the following: $n=[k]_q$ and $P_{1},\dots, P_{n}$ are mutually distinct.
In this case, $C$ is called a {\it $q$-ary simplex code of dimension $k$}.
Under the assumption that $n=[k]_q$ the following three conditions are equivalent:
\begin{enumerate}
\item[$\bullet$] $C$ is a simplex code.
\item[$\bullet$] there is a one-to-one correspondence between $P_{1},\dots,P_{n}$
and the points of the projective geometry ${\rm PG}(k-1,q)$, in other words,
the projective system of $C$ is ${\rm PG}(k-1,q)$.
\item[$\bullet$] All columns in a generator matrix of $C$ are mutually non-proportional
(if this condition holds for a certain generator matrix, then it holds for all generator matrices).
\end{enumerate}
Note that $P_{i}\ne P_{j}$ is equivalent to the fact that $C\cap C_{i}$ and $C\cap C_{j}$ are distinct hyperplanes of $C$.
Since $C$ contains precisely $[k]_q=n$ hyperplanes, 
every hyperplane of $C$ is $C\cap C_{i}$.
Recursively, we establish that 
every $m$-dimensional subspace of $C$ is the intersection of precisely $[k-m]_q$ distinct $C\cap C_i$.
Taking $m=1$ we obtain that
every non-zero codeword $x\in C$ has precisely $[k-1]_q$ coordinates equal to $0$.

Recall that a map $l:V\to V$ is called {\it semiliear} if $$l(x+y)=l(x)+l(y)$$ 
for all $x,y\in V$ and there is an automorphism $\sigma$ of ${\mathbb F}_q$ such that
$$l(ax)=\sigma(a)l(x)$$ for all $x\in V$ and $a\in {\mathbb F}_q$
(if $\sigma$ is identity, then $l$ is linear). 
Note that non-identity automorphisms of ${\mathbb F}_{q}$ exist only in the case when $q$ is not a prime number. 
A semilinear automorphism of $V$ is said to be {\it monomial} if it sends every $e_i$ to a non-zero scalar multiple of $e_{j}$.

An {\it isomorphism} between non-degenerate codes $C,C'\subset V$ is a semilinear isomorphism of $C$ to $C'$
which can be extended to a monomial semilinear  automorphism of $V$. 
If one of these codes is simplex, then the same holds for the other and such an extension is unique. 
Indeed, if the restriction of a monomial semilinear automorphism of $V$ to a simplex code $C$ is identity,
then it transfers every $e_i$ to $a_{i}e_{i}$ (this follows from the fact that every hyperplane of $C$ is $C\cap C_{i}$
and the hyperplanes $C\cap C_i$, $C\cap C_j$ are distinct if $i\ne j$) and it is clear that each $a_i$ is $1$.

For every code isomorphism $v:C\to C'$ the contragradient ${\check v}: C^*\to C'^*$ (the inverse of the adjoint map $v^{*}$ \cite[Section 1.7]{Pank-book})
sends ${\mathcal P}(C)$ to ${\mathcal P}(C')$.
Conversely, if a semilinear isomorphism $u:C^*\to C'^*$ transfers ${\mathcal P}(C)$ to ${\mathcal P}(C')$,
then ${\check u}:C\to C'$ is a code isomorphism  \cite[Lemma 2.20]{Pank-book}.

It was noted above that $q$-ary simplex codes of dimension $k$ can be characterized as codes whose projective systems are ${\rm PG}(k-1,q)$.
This implies that the group of monomial linear automorphisms of $V$ acts transitively on the set of all simplex codes, i.e. any two such codes are equivalent. 
The automorphism group of any simplex code $C\subset V$ is isomorphic to $\Gamma {\rm L}(k,q)$
and the group of linear automorphisms is isomorphic to ${\rm GL}(k,q)$.
Since the group of monomial linear automorphisms of $V$ and the group ${\rm GL}(k,q)$ are of orders 
$$n!(q-1)^n\;\mbox{ and }\;(q^{k}-1)(q^k -q)\dots(q^{k}-q^{k-1})$$
(respectively), there are precisely
$$\frac{n!(q-1)^n}{(q^{k}-1)(q^k -q)\dots(q^{k}-q^{k-1})}$$
 $q$-ary simplex codes of dimension $k$.

We say that $x\in V$ is a {\it simplex vector} if this vector has precisely $[k-1]_q$ coordinates equal to $0$.
A non-zero vector is simplex if and only if it is a codeword of a certain simplex code.
Indeed, all codewords of simplex codes are simplex vectors (it was noted above)
and the group of monomial linear automorphisms of $V$ acts transitively on the set of all simplex vectors and the set of all simplex codes.
A subspace of $V$ is a simplex code 
if and only if it is maximal with respect to the property that all non-zero vectors are simplex 
(see \cite{B} or \cite[Theorem 7.9.5]{HP-book}).
By \cite{KPP}, the set of all simplex vectors $x=(x_{1},\dots, x_{n})$ can be characterized by the following equations
$$\sum_{i_1<\dots< i_{p^{j}}}x_{i_1}^{q-1}\dots\; x_{i_{p^{j}}}^{q-1}  =  0$$
for $j\in\{0,\dots,mk-m-1\}$, where $q=p^m$ and $p$ is a prime number.

\section{The graphs of $2$-dimensional simplex codes}
In this section, we suppose that $k=2$. Then $n=q+1$ and a vector of $V$ is simplex if and only if precisely one of its coordinates  is zero.
For a non-zero vector $(a_{1},\dots,a_{n})\in V$ we denote by $\langle a_{1},\dots, a_{n}\rangle$
the $1$-dimensional subspace containing this vector.
We will use the projective geometry language and
call $1$-dimensional and $2$-dimensional subspaces of $V$ {\it points} and {\it lines}, respectively;
in particular, $1$-dimensional subspaces containing simplex vectors are said to be {\it simplex points}.
Each line consists of $n=q+1$ points.
Simplex codes correspond to lines formed by simplex points and every such line will be called {\it simplex}.
Every point of a simplex line is the intersection of this line with a unique coordinate hyperplane.

\begin{exmp}\label{exmp1}{\rm
Suppose that $q=4$.
Consider the line containing the simplex points 
$$\langle 0,1,1,1,1\rangle\;\mbox{ and }\;\langle 1,0,1,\alpha, \alpha^2 \rangle,$$
where $\alpha$ is a primitive element of ${\mathbb F}_4$.
The remaining three points on this line are
$$\langle 1,1,0,\alpha^2,\alpha\rangle,\;\langle 1,\alpha,\alpha^2,0,1\rangle,\;
\langle 1,\alpha^2, \alpha, 1,0 \rangle$$
and the line is simplex.
}\end{exmp}

We denote by $\Gamma$ the restriction of the Grassmann graph $\Gamma_{2}(V)$
to the set of all simplex lines and say that two simplex lines are {\it adjacent} if they are adjacent vertices of the graph $\Gamma$.
For some pairs of distinct simplex points the line containing them is not simplex,
two distinct simplex points are said to be  {\it adjacent} if they are connected by a simplex line. 

\begin{prop}\label{prop-element}
The following assertions are fulfilled:
\begin{enumerate}
\item[(1)] Every simplex point is contained in precisely $(q-1)!$ simplex lines.
\item[(2)] The degree of every vertex of the graph $\Gamma$ is equal to $(q+1)[(q-1)! -1]$.
\end{enumerate}
\end{prop}

\begin{proof} (1).
Let $P=\langle 0,1,\dots,1\rangle$.
A simplex point $\langle 1,a_{1},\dots, a_{q}\rangle$ is adjacent to $P$ 
if and only if all columns of the matrix
$$\left[\begin{array}{c@{\;}c@{\;}c@{\;}c}
0&1&\dots&1\\
1&a_{1}&\dots& a_{q}\\
\end{array}
\right]$$
are mutually non-proportional; 
the latter is equivalent to the fact that $a_{1},\dots,a_{q}$ are mutually distinct. 
Therefore, there are precisely $q!$ simplex points adjacent to $P$.
Every simplex line passing through $P$ contains $q$ points distinct from $P$
which means that $P$ is contained in precisely  $(q-1)!$ simplex lines.
Since the group of monomial linear automorphisms of $V$ acts transitively on 
the sets of simplex points and the set of simplex lines,
the above holds for every simplex point. 

(2). Each simplex line $L$ consists of $q+1$ simplex points.
Every such point is contained in precisely $(q-1)! -1$ simplex lines distinct from $L$.
\end{proof}

\begin{exmp}\label{exmp-grid}{\rm
If $q=3$, then $n=4$.
In this case, the set of all simplex vectors $x=(x_{1},x_{2},x_{3},x_{4})$
is described by the quadratic equation 
$$x^{2}_{1}+x^{2}_{2}+x^{2}_{3}+x^{2}_{4}=0$$
and all simplex points together with all simplex lines form a generalized quadrangle.
There are precisely $16$ simplex points and each of them is contained in precisely two simplex lines.
There are precisely $8$ simplex lines which  form  a grid and $\Gamma$ is isomorphic to $K_{4,4}$.
}\end{exmp}

From this moment, we suppose that $q\ge 4$.

\begin{prop}\label{prop-ad}
Let $P=\langle x_{1},\dots,x_{q+1}\rangle$ be a simplex point and $x_{i}=0$.
If a simplex point $Q=\langle y_{1},\dots, y_{q+1}\rangle $ is adjacent to $P$, 
then
\begin{equation}\label{eq-ad}
a_{1}y_{1}+\dots+ a_{q+1}y_{q+1}=0,\;\mbox{ where }\; a_{j}=x^{-1}_j\;\mbox{ for }\;j\ne i\;\mbox{ and }\;a_i=0\,.
\end{equation}
In the case when $q=4$, a simplex point $Q=\langle y_{1},\dots, y_{q+1}\rangle $ is adjacent to $P$
if and only if it satisfies \eqref{eq-ad} and  $y_{i}\ne 0$.
\end{prop}

\begin{proof}
The simplex points $P,Q$ are adjacent  if and only if $y_{i}\ne 0$
and the columns of the matrix
$$\left[\begin{array}{c@{\;}c@{\;}c}
x_{1}&\dots&x_{q+1}\\
y_{1}&\dots&y_{q+1}\\
\end{array}
\right]$$
are mutually non-proportional. 
The latter is equivalent to the fact that $y_i\ne 0$ and
$$x^{-1}_{1}y_{1},\dots,x^{-1}_{i-1}y_{i-1},\,x^{-1}_{i+1}y_{i+1},\dots ,x^{-1}_{q+1}y_{q+1}$$ 
are mutually distinct elements of ${\mathbb F}_{q}$
(one of them is zero).
The sum of all $q-1$ non-zero elements of ${\mathbb F_{q}}$ is zero 
and we obtain \eqref{eq-ad}. 

In the case when $q=4$, the sum of three non-zero elements of ${\mathbb F}_q$ is zero 
if and only if these elements are mutually distinct. This implies the second statement.
\end{proof}

\begin{rem}\label{rem-ad}{\rm
If $q=4$ and $T$ is a simplex point contained in  the hyperplane defined by \eqref{eq-ad}
and non-adjacent to $P$ (it will be shown latter that such points exist),
then Proposition \ref{prop-ad} implies that $P$ and $T$ are contained in the same coordinate hyperplane.
In the case when $q\ge 5$, 
there exist $b_{1},\dots,b_{q-1}\in{\mathbb F}_{q}\setminus\{0\}$ 
such that $b_{1}+\dots+b_{q-1}=0$ and $b_{i}=b_{j}$ for some distinct $i,j$; 
therefore, the hyperplane defined by \eqref{eq-ad} contains 
simplex points $T$ non-adjacent to $P$ and such that $P,T$ belong to distinct coordinate hyperplanes.
}\end{rem}

The non-empty intersection of a maximal clique of $\Gamma_2(V)$ (a star or a top) with the set of simplex lines 
is a clique in $\Gamma$, but we cannot state that this is a maximal clique of $\Gamma$.

\begin{prop}\label{prop-cliq}
If $q=4$, then every maximal clique of $\Gamma$ consists of all simplex lines passing through a certain simplex point. 
\end{prop}

\begin{proof}
It is sufficient to show that for $q=4$ there exist no three mutually adjacent simplex lines intersecting in three distinct points.
Suppose that $S_1, S_2, S_3$ are such lines and
$$S_1\cap S_2=\langle x\rangle, \; S_1\cap S_3=\langle y\rangle, \; S_2\cap S_3=\langle z\rangle$$
are mutually distinct and non-collinear.
Without lose of generality we can assume that $S_1$ is the simplex line considered in Example \ref{exmp1} and 
$$x=(0,1,1,1,1),\;\;y=(1,0,1,\alpha,\alpha^2).$$
Indeed, the group of monomial linear automorphisms of $V$ acts transitively on the set of simplex lines
and for any two distinct points on a simplex line there is an automorphism of this line sending this pair to any other pair of points on the line.
Applying Proposition \ref{prop-ad} to the pairs $\langle x \rangle, \langle z \rangle$ and $\langle y \rangle, \langle z \rangle$,
we establish that $z$ is a scalar multiple of one of the following vectors
$$(1,1,0,\alpha^2,\alpha),\; (1,\alpha,\alpha^2,0,1),\; (1,\alpha^2,\alpha,1,0);$$
but the corresponding simplex points belong to the line $S_1$ and we get a contradiction.
\end{proof}

\begin{exmp}{\rm 
Proposition \ref{prop-cliq} fails for $q\ge 5$.
If $q=5$, then the simplex vectors
$$x=(0,1,1,1,1,1),\; y=(1,0,1,2,4,3),\; z=(4,3,1,4,2,0)$$
define non-collinear simplex points and the lines
$$\langle x\rangle + \langle y\rangle, \;\langle x\rangle + \langle z\rangle, \; \langle y \rangle + \langle z\rangle$$
are simplex.
A similar example can be constructed for every $q\ge 5$.
}\end{exmp}

The following statement completely describe the graph $\Gamma$ in the case  when $q=4$. 

\begin{theorem}\label{theorem1}
Let $q=4$. Then $\Gamma$ is a connected graph of diameter $3$ consisting of $162$ simplex lines
and the degree of every vertex of $\Gamma$ is equal to $25$.
For every simplex line $L$ the following assertions are fulfilled:
\begin{enumerate}
\item[(1)] There are precisely $6$ simplex lines $L_{1},\dots,L_{6}$ which are at distance $3$ from $L$ in the graph $\Gamma$.
\item[(2)] There are precisely $130$ simplex lines which are at distance $2$ from $L$ in the graph $\Gamma$.
The set of all such lines is the union of three mutually disjoint subsets denoted by ${\mathcal X}^{3}_{20},{\mathcal X}^{1}_{90}, {\mathcal X}^{0}_{20}$,
where
\begin{enumerate}
\item[$\bullet$]${\mathcal X}^{3}_{20}$ is formed by $20$ simplex lines and each of these lines is adjacent to precisely three distinct $L_{i}$, 
\item[$\bullet$]${\mathcal X}^{1}_{90}$ consists of  $90$ lines and every such  line is adjacent to a unique $L_{i}$,
\item[$\bullet$]${\mathcal X}^{0}_{20}$ consists of $20$ simplex lines disjoint with all $L_i$.
\end{enumerate}
\item[(3)] $\{L,L_{1},\dots, L_{6}\}\cup {\mathcal X}^{0}_{20}$ is a spread of the set of all $135$ simplex points,
i.e. this set consists of $27$ mutually disjoint lines which cover the set of simplex points.  
\end{enumerate}
\end{theorem}

As above, we suppose that $q=4$.
In this case, there is the unique non-identify field automorphism which sends every $a$ to $a^2$.
Semilinear automorphisms of $V$ induces bijective transformations of the associated projective geometry.
Such transformations are called {\it projective}. 
Two semilinear automorphisms induce the same projective transformation if and only if one of them is a scalar multiple of the other.
For every simplex line $L$ we denote by ${\rm G}(L)$ the group of all projective transformations induced by 
monomial semilinear automorphisms of $V$ preserving $L$, i.e. the extensions of automorphisms of the corresponding simplex code
(recall that every such extension is unique).
The group ${\rm G}(L)$ is isomorphic to ${\rm P}\Gamma {\rm L}(2,4)$ 
(since the automorphism group of the corresponding code is isomorphic to $\Gamma {\rm L}(2,4)$)
and the subgroup of ${\rm G}(L)$ formed by transformations induced by linear automorphisms is isomorphic to ${\rm PGL}(2,4)$.
On the other hand, ${\rm P}\Gamma {\rm L}(2,4)$ is a subgroup of the permutation group $S_{5}$ acting on the points of $L$.
These groups both are of order $120$ which means that they are coincident. 
Also, ${\rm PGL}(2,4)$ is of order $60$ and the alternating group $A_5$ is the unique subgroup of $S_5$ whose order is $60$.
Therefore, every permutation of the points of $L$ can be uniquely extended to an element of ${\rm G}(L)$;
moreover, an element of ${\rm G}(L)$ is induced by a linear automorphism if and only if it gives an even permutation of the points of $L$.

\begin{theorem}\label{theorem2}
Suppose that $q=4$ and $L$ is a simplex line. Let $L_{1},\dots, L_{6}$ and ${\mathcal X}^{3}_{20},{\mathcal X}^{1}_{90}, {\mathcal X}^{0}_{20}$
be as in Theorem \ref{theorem1} and let ${\mathcal A}$ be the set of all simplex lines adjacent to $L$.
The action of the group ${\rm G}(L)$ on the set of all simplex lines has the following properties:
\begin{enumerate}
\item[{\rm (1)}]
The sets $\{L_{1},\dots, L_{6}\}$, ${\mathcal X}^{3}_{20}$ and ${\mathcal X}^{0}_{20}$ are orbits of this action;
moreover, the action of ${\rm G}(L)$ on the set $\{L_{1},\dots, L_{6}\}$ is sharply $3$-transitive. 
\item[{\rm (2)}] The set ${\mathcal A}$ is the union of two orbits consisting of $10$ and $15$ simplex lines. 
\item[(3)] The set ${\mathcal X}^{1}_{90}$ is the union of two orbits consisting of $30$ and $60$ simplex lines. 
\end{enumerate}
\end{theorem}

\begin{rem}{\rm
Two linear codes with the same parameters are adjacent if and only if the dual codes are adjacent.
Simplex codes are dual to Hamming codes. Recall that under our assumption $k=2$ and $n=q+1$.
The graph $\Gamma$ is isomorphic to the restriction of the Grassmann graph $\Gamma_{q-1}(V)$ to the set of $q$-ary Hamming codes of dimension $q-1$.
Therefore, Theorems \ref{theorem1} and \ref{theorem2} can be reformulated in terms of $4$-ary Hamming codes of dimension $3$.
}\end{rem}

\section{Proof of Theorems \ref{theorem1} and \ref{theorem2}}
Suppose that $k=2$ and $q=4$. Then $n=5$.

\subsection{Six simplex lines}
Let $L$ be a simplex line. 
Denote by $P_{i}$ the simplex point which is the intersection of $L$ with the coordinate hyperplane $C_i$.
The line $L$ consists of the points $P_{1},\dots,P_5$.
Let $H_{i}$ be the hyperplane containing all simplex points adjacent to $P_i$ and described by the equality from Proposition \ref{prop-ad}.
A simplex point $P\ne P_i$ from $H_i$ is not adjacent to $P_i$ if and only if it is contained in $H_i\cap C_i$. 
\begin{lemma}\label{lemma-d3}
For a simplex line $L'$ the following two conditions are equivalent:
\begin{enumerate}
\item[(1)] $L'$ intersects every $H_i$ in a point contained in  $H_i\cap C_i$ and distinct from $P_i$,
\item[(2)]  the distance between $L$ and $L'$ in the graph $\Gamma$ is greater than $2$ or $L$ and $L'$ cannot be connected in $\Gamma$.
\end{enumerate}
\end{lemma}

\begin{proof}
The condition (1) holds if and only if $L'$ is disjoint with all simplex lines adjacent to $L$.
\end{proof}

It will be shown latter that there are precisely $6$ simplex lines $L'$ satisfying the condition (1) from Lemma \ref{lemma-d3}
and these lines are adjacent to some of the remaining simplex lines. 
This implies that $\Gamma$ is a connected graph of diameter $3$.

\begin{exmp}\label{exmp-main}{\rm
Suppose that $L$ is the simplex line considered in Example 1.
We identify this line with the matrix 
$$\left[\begin{array}{c@{\;}c@{\;}c@{\;}c@{\;}c}
0&1&1&1&1\\
1&0&1&\alpha&\alpha^2\\
1&1&0&\alpha^2&\alpha\\
1&\alpha&\alpha^2&0&1\\
1&\alpha^2&\alpha&1&0\\
\end{array}
\right],$$
where the $i$-th row corresponds to a vector belonging to $P_i$.
By Proposition \ref{prop-ad}, the hyperplanes $H_1,\dots,H_5$ are defined by the following equations:
\begin{enumerate}
\item[(H1)] $x_2+x_{3}+x_{4}+x_{5}=0$,
\item[(H2)] $x_{1}+x_{3}+\alpha^2 x_{4} +\alpha x_5=0$,
\item[(H3)] $x_{1}+x_{2}+\alpha x_{4} +\alpha^2 x_5=0$,
\item[(H4)] $x_{1}+\alpha^2 x_2+ \alpha x_3 + x_5=0$,
\item[(H5)] $x_{1}+\alpha x_2+ \alpha^2 x_3 +x_{4}=0$,
\end{enumerate}
Also, $H_{1}\cap C_{1}$ contains precisely $6$ simplex points:
$$Q_1=\langle 0, 1,1,\alpha,\alpha\rangle,\; Q_2=\langle 0, 1,\alpha ,1, \alpha\rangle,\; Q_3=\langle 0, 1,\alpha,\alpha,1\rangle,$$
$$Q_4=\langle 0, 1,1,\alpha^2,\alpha^2\rangle,\;Q_5=\langle 0, 1,\alpha^2,1,\alpha^2\rangle,\;Q_6=\langle 0,1,\alpha^2,\alpha^2,1\rangle.$$
}\end{exmp}

\begin{rem}\label{rem-Q6}{\rm
We describe some simple properties of the simplex points $Q_1,\dots, Q_{6}$ which will be used in what follows.
It is clear that the line joining distinct $Q_i,Q_{j}$ is not simplex.
A direct verification shows that the line connecting $Q_1$ and $Q_{i}$, $i\ge 2$ does not contain any $Q_{k}$ with $k\ne 1,i$
and $P_{1}$ belongs to this line if and only if $i=4$.
Similarly, we establish that the line connecting  $Q_{i}$ and $Q_{j}$, $i<j$ does not contain any $Q_{k}$ with $k\ne i,j$
and $P_{1}$ belongs to this line only in the case when $j=i+3$.
In particular, any three distinct $Q_i$ are non-collinear.
}\end{rem}

\begin{rem}\label{rem3}{\rm
Since the group of monomial linear automorphisms of $V$ acts transitively on the sets of simplex points and simplex lines,
for every simplex point $P$ the hyperplane containing all simplex points adjacent to $P$ 
contains also precisely $6$ simplex points non-adjacent to $P$. 
}\end{rem}

By the symmetry, it is sufficient to prove Theorems \ref{theorem1} and \ref{theorem2} for the case when 
$L$ is the line considered in Examples 1 and \ref{exmp-main}.
In what follows, we will always assume that $L$ is this line.

Every monomial linear automorphism of $V$ can be presented as the composition 
$$d(a_{1},\dots,a_{5})p_{\sigma},$$
where $d(a_{1},\dots,a_{5})$ is the linear automorphism whose matrix with respect to the standard basis $e_{1},\dots,e_{5}$ is ${\rm diag}(a_{1},\dots,a_{5})$
and $p_{\sigma}$ is the linear automorphism sending every $e_i$ to $e_{\sigma(i)}$ for a given permutation $\sigma$ on the set $\{1,\dots,5\}$.

\begin{lemma}\label{lemma-q6}
Let $L$ and $Q_{1},\dots,Q_{6}$ be as in Example \ref{exmp-main}.
Then for any two distinct $i,j\in \{1,\dots,6\}$ there is a monomial linear automorphism of $V$
preserving $L$ and sending $Q_{i}$ to $Q_{j}$.
\end{lemma}

\begin{proof}
The linear automorphism $$d(\alpha^2,1,1,1,1)p_{(3,4,5)}$$ preserves 
$L$ (since it leaves fixed $P_1$ and $P_2$) and induces on the set $\{Q_1,\dots,Q_{6}\}$
the permutation 
$$(Q_{1},Q_{2},Q_{3})(Q_{4},Q_{5},Q_{6}).$$
The linear automorphism $p_{(2,5)(3,4)}$
preserves $L$ (since it transposes $P_{2}$ and $P_{5}$) and induces on the set $\{Q_1,\dots,Q_{6}\}$
the permutation $(Q_1,Q_4)(Q_2,Q_5)$.
\end{proof}

\begin{lemma}\label{lemma-q6-2}
For every $i\in \{1,\dots,6\}$ there is a unique simplex line $L_i$ passing through $Q_i$ and intersecting every $H_j\cap C_j$ in a point distinct from $P_j$.
\end{lemma}

\begin{proof}
Let $x=(0,1,1,\alpha,\alpha)$.
Suppose that $L_{1}$ is a simplex line containing the point $Q_{1}=\langle x \rangle$  and intersecting each $H_{i}\cap C_{i}$
in a point $P'_{i}$ distinct from $P_i$. Then $P'_1=Q_1$ and $P'_1,\dots,P'_5$ are mutually distinct 
(if $P'_i=P'_j$ for some distinct $i,j$, then $P'_i=P'_j \subset C_i\cap C_j$ which contradicts the fact that $P'_i=P'_j$ is a simplex point).

There is a unique vector $y\in P'_2$ such that $x+y\in P'_{3}$.
The third coordinate of $x+y$ is $0$. Hence the third coordinate of $y$ is $1$, i.e.
$$y=(a,0,1,b,c).$$ 
Since the simplex points $Q_{1}=\langle x \rangle$ and $P'_2=\langle y \rangle$ are adjacent,
Proposition \ref{prop-ad} shows that 
$$1+\alpha^2 b+ \alpha^2 c=0.$$
Also, $y$ satisfies the equation (H2) from Example \ref{exmp-main} and we have
$$a+1+\alpha^2 b +\alpha c=0.$$
Joining these two equalities, we obtain that $a=c$. 
It follows from the first equality that $1,\alpha^2b, \alpha^2 c$ are mutually  distinct non-zero elements of the field
and one of the following possibilities is realized:
$$y=(\alpha^2, 0,1,1,\alpha^2)\;\mbox{ or }\;y=(1,0,1, \alpha^2,1).$$
Then
$$x+y=(\alpha^2, 1,0, \alpha^2,1)\;\mbox{ or }\; x+y=(1,1,0,1,\alpha^2),$$
respectively, 
and each of these vectors satisfies (H3) from Example \ref{exmp-main}.
So, we have determined  two possibilities for the triple of points which are the intersections of $L_1$ with $C_1,C_2,C_3$.
The remaining two points of  $L_1$ are defined by the vectors
$$\alpha^2 x+y= (\alpha^2, \alpha^2, \alpha, 0, \alpha),\; \alpha x+y= (\alpha^2, \alpha, \alpha^2, \alpha,0)$$
or the vectors
$$\alpha x+ y= (1,\alpha, \alpha^2, 0, \alpha), \; \alpha^2 x +y= (1,\alpha^2, \alpha,\alpha,0).$$
The fist two vectors satisfy (H4) and (H5), respectively. 
However, these equations do not hold for the second pair of vectors.
Therefore, there is a unique simplex line $L_1$ passing through $Q_1$ and intersecting each $H_{i}\cap C_{i}$
in a point distinct from $P_i$. This line consists of $Q_1$ and the following four points 
$$\langle \alpha^2, 0,1,1,\alpha^2 \rangle,\; \langle \alpha^2, 1,0, \alpha^2,1\rangle,\;
\langle \alpha^2, \alpha^2, \alpha, 0, \alpha \rangle,\; \langle \alpha^2, \alpha, \alpha^2, \alpha,0 \rangle.$$
By Lemma \ref{lemma-q6}, for every $i\in \{2,\dots,6\}$
there is a monomial linear automorphism $u$ preserving $L$ and sending $Q_{1}$ to $Q_{i}$.
Then $L_{i}=u(L_{1})$ is a simplex line passing through $Q_{i}$ and intersecting every $H_j\cap C_j$ in a point distinct from $P_j$.
If $L'$ is a simplex line satisfying the same conditions, then 
the simplex line $u^{-1}(L')$  passes through $Q_1$ and intersects every $H_j\cap C_j$ in a point distinct from $P_j$
which implies that $u^{-1}(L')=L_1$ and $L'=u(L_{1})=L_i$.
\end{proof}

Let ${\rm G}$ be the group generated by the linear automorphisms described in the proof of Lemma \ref{lemma-q6}.
This group acts transitively  on the set $\{Q_1,\dots, Q_6\}$.
For every $i\ge 2$ the line $L_i$ coincides with $u(L_1)$, where $u$ is any element of ${\rm G}$ sending $Q_1$ to $Q_i$.
Each $L_i$ intersects every $H_j\cap C_j$ in a $1$-dimensional subspace distinct from $P_j$
and we present $L_i$ as the matrix whose $j$-th row corresponds to a non-zero vector belonging to this $1$-dimensional subspace
(we always choose the vector whose first non-zero coordinate is $1$):
$$L_1=\left[\begin{array}{c@{\;}c@{\;}c@{\;}c@{\;}c}
0&1&1&\alpha&\alpha\\
1&0&\alpha&\alpha&1\\
1&\alpha&0&1&\alpha\\
1&1&\alpha^2&0&\alpha^2\\
1&\alpha^2&1&\alpha^2&0\\
\end{array}
\right],\,
L_2=
\left[\begin{array}{c@{\;}c@{\;}c@{\;}c@{\;}c}
0&1&\alpha&1&\alpha\\
1&0&\alpha&\alpha^2&\alpha^2\\
1&1&0&\alpha&1\\
1&\alpha^2&\alpha^2&0&\alpha\\
1&\alpha&1&1&0\\
\end{array}
\right],\,
L_3=
\left[\begin{array}{c@{\;}c@{\;}c@{\;}c@{\;}c}
0&1&\alpha&\alpha&1\\
1&0&1&\alpha^2&1\\
1&\alpha^2&0&\alpha&\alpha\\
1&\alpha&\alpha&0&\alpha^2\\
1&1&\alpha^2&1&0\\
\end{array}
\right],
$$
$$L_4=\left[\begin{array}{c@{\;}c@{\;}c@{\;}c@{\;}c}
0&1&1&\alpha^2&\alpha^2\\
1&0&\alpha^2&1&\alpha^2\\
1&\alpha^2&0&\alpha^2&1\\
1&\alpha&1&0&\alpha\\
1&1&\alpha&\alpha&0\\
\end{array}
\right],\,
L_5=
\left[\begin{array}{c@{\;}c@{\;}c@{\;}c@{\;}c}
0&1&\alpha^2&1&\alpha^2\\
1&0&1&1&\alpha\\
1&\alpha&0&\alpha^2&\alpha^2\\
1&1&\alpha&0&1\\
1&\alpha^2&\alpha^2&\alpha&0\\
\end{array}
\right],\,
L_6=
\left[\begin{array}{c@{\;}c@{\;}c@{\;}c@{\;}c}
0&1&\alpha^2&\alpha^2&1\\
1&0&\alpha^2&\alpha&\alpha\\
1&1&0&1&\alpha^2\\
1&\alpha^2&1&0&1\\
1&\alpha&\alpha&\alpha^2&0\\
\end{array}
\right]\,.
$$
A direct verification shows that any two distinct $L_{i},L_{j}$ are disjoint, 
i.e. they span a hyperplane which will be denoted by $S_{ij}$.

\begin{lemma}\label{lemma-Sij}
The hyperplane $S_{ij}$ does not contain any $L_{k}$ with $k\ne i,j$.
\end{lemma}

\begin{proof}
Without loss of generality we can assume that $i=1$ (since the group ${\rm G}$ acts transitively on the set $\{L_1,\ldots ,L_6\}$).
Suppose that $S_{1j}$ contains a certain $L_{k}$ with $k\ne 1,j.$
Then $S_{1j}$ contains the points $Q_{1},Q_{j},Q_{k}$. 
By Remark \ref{rem-Q6}, these points are non-collinear and span the projective plane $H_{1}\cap C_1$
which means that $H_{1}\cap C_1\subset S_{1j}$. 
Every (projective) point $P\subset S_{1j}\cap C_1$ is contained in $H_1$
(otherwise $P$ and $H_{1}\cap C_1$ span $S_{1j}$ and we obtain that $S_{1j}$ coincides with $C_1$ which is impossible).
Therefore, 
$$S_{1j}\cap C_1=H_{1}\cap C_1.$$
Let $x_{1t}$ be the sum of the second rows from the matrices corresponding to $L_1$ and $L_t$.
Then the points $\langle x_{1j}\rangle $ and $\langle x_{1k}\rangle$ are contained in $S_{1j}$ (since $L_1,L_j,L_k\subset S_{1j}$). 
All the vectors 
$$x_{12}=(0,0,0,1,\alpha),\; x_{13}=(0,0,\alpha^2,1,0),\; x_{14}=(0,0,1,\alpha^2,\alpha),$$
$$x_{15}=(0,0,\alpha^2,\alpha^2,\alpha^2),\; x_{16}=(0,0,1,0,\alpha^2),$$
belong to $C_1$,
but only $x_{14}$ is contained in $H_1$ (recall that $H_1$ is defined by the equation $x_2+x_{3}+x_{4}+x_{5}=0$).
So, there is no pair of distinct $j,k$ such that both $\langle x_{1j}\rangle $ and $\langle x_{1k}\rangle$ are contained in $S_{1j}$
and we get a contradiction.
\end{proof} 

\begin{lemma}\label{lemma-q}
For every simplex point $Q$ contained in a certain $H_i\cap C_{i}$ and distinct from $P_i$
there is a unique $L_{j}$ passing through $Q$.
\end{lemma}

\begin{proof}
Denote by ${\mathcal X}$ the set of all simplex points $Q$ such that $Q$ is contained in a certain $H_i\cap C_{i}$ and distinct from $P_i$.
Observe that the intersection of $H_{i}\cap C_{i}$ and $H_{j}\cap C_{j}$ does not contain simplex points if $i\ne j$.
Since every $H_i\cap C_{i}$ contains precisely $6$ simplex points distinct from $P_i$ (Remark \ref{rem3}), 
the set ${\mathcal X}$ consists of $5\cdot 6=30$ points. 
Every $L_{i}$ is formed by $5$ points from the set ${\mathcal X}$ and
the statement (1) follows from the fact that the lines $L_{1},\dots, L_6$ are mutually disjoint. 
\end{proof}

\begin{prop}\label{prop1-conn}
$\Gamma$ is a connected graph of diameter $3$. 
The distance between $L$ and a simplex line $L'$ is equal to $3$ if and only if $L'\in \{L_1,\dots,L_6\}$.
\end{prop}

\begin{proof}
The condition (1) of Lemma \ref{lemma-d3} holds for every $L'\in \{L_{1},\dots,L_{6}\}$.
Conversely, if $L'$ satisfies this condition, then $L'$ intersects $H_1\cap C_1$ in a certain $Q_{j}$ and Lemma \ref{lemma-q6-2} implies that  $L'\in \{L_{1},\dots,L_{6}\}$.
Therefore, a simplex line $L'$ satisfies the condition (1) of Lemma \ref{lemma-d3} if and only if $L'\in \{L_{1},\dots,L_{6}\}$.

The required statement follows from the fact that every $L_i$ is adjacent to a certain simplex line $L'\not\in \{L_{1},\dots,L_{6}\}$. 
\end{proof}

\begin{lemma}\label{lemma-L'}
Every simplex line $L'$ non-adjacent to $L$ intersects the hyperplanes $H_1,\dots,H_5$ in mutually distinct points.
\end{lemma}

\begin{proof}
For the lines $L_{1},\dots,L_{6}$ it was established above (see the proof of Lemma \ref{lemma-q6-2}).
Consider the general case.
Suppose that the statement fails and $L'$ contains a point $Q$ belonging to distinct $H_i$ and $H_j$.
The point $Q$ is contained in a unique coordinate hyperplane $C_k$ (as a simplex point).
If $k=i$, then $Q$ is in $H_i\cap C_i$ which means that it belong to a certain $L_t$. 
Since the intersection of $H_{i}\cap C_{i}$ and $H_{j}\cap C_{j}$ does not contain simplex points,
$Q$ is adjacent to $P_{j}$. 
Then the simplex lines $L_t$ and $Q+P_{j}$ are adjacent  and the distance between $L_{t}$ and $L$ is not greater than $2$,
a contradiction.   
The case when $k=j$ is similar.
So, $k$ is distinct from $i$ and $j$.
Then $Q$ is adjacent to both $P_i, P_j$ and the simplex lines $Q+P_i,Q+P_j,L$ form a clique of $\Gamma$ which is impossible by Proposition \ref{prop-cliq}.
\end{proof}

Let $L'$ be a simplex line distinct from $L$ and non-adjacent to it.
Then $L'$ intersects each $H_j$ in a point distinct from $P_j$.
Denote by  $n(L')$ the number of all indices $j$ such that $L'$ intersects $H_j$ in a point contained in $H_j\cap C_{j}$.
If $L'\in \{L_1,\dots,L_{6}\}$, then $n(L')=5$.
In the case when $L'\not\in \{L_1,\dots,L_{6}\}$, 
Lemma \ref{lemma-q} shows that $n(L')$ is the number of $L_i$ intersecting $L'$.

\begin{lemma}\label{lemma-nL}
If $L'$ is distinct from $L_1,\dots,L_6$,
then $n(L')$ is equal to  $0,1$ or $3$. 
\end{lemma}

\begin{proof}
Let $P'_i$ be the intersection of $L'$ with the hyperplane $H_i$.
By Lemma \ref{lemma-L'}, $P'_1,\dots,P'_5$ are mutually distinct.
Since $L$ and $L'$ are disjoint, $P'_{i}$ and $P_{j}$ are distinct for all $i,j\in \{1,\dots,5\}$.

The hyperplanes $H_1,\dots,H_5$ are defined by the following equations:
\begin{enumerate}
\item[(H1)] $x_2+x_{3}+x_{4}+x_{5}=0$,
\item[(H2)] $x_{1}+x_{3}+\alpha^2 x_{4} +\alpha x_5=0$,
\item[(H3)] $x_{1}+x_{2}+\alpha x_{4} +\alpha^2 x_5=0$,
\item[(H4)] $x_{1}+\alpha^2 x_2+ \alpha x_3 + x_5=0$,
\item[(H5)] $x_{1}+\alpha x_2+ \alpha^2 x_3 +x_{4}=0$,
\end{enumerate}
see Example \ref{exmp-main}.
In what follows the equation ${\rm (Hi)}$, ${\rm i}=1,\dots,5$ will be denoted as 
$$A_{i}(x)=0,\;\;x=(x_{1},\dots,x_5).$$
Observe that
$$A_3=A_1+A_2,\;\;A_4=\alpha^2 A_1+A_2,\;\;A_5=\alpha A_1+A_2.$$
For any vectors 
$$x=(x_{1},\dots,x_{5})\in P'_1\;\mbox{ and }\;y=(y_{1},\dots,y_{5})\in P'_2$$
we have $A_{1}(x)=A_{2}(y)=0$.
We choose $x$ and $y$ satisfying 
$$A_{1}(y)=A_{2}(x)$$
and obtain that
$$x+y\in P'_3,\;\; \alpha^2 x+y\in P'_4,\;\;\alpha x +y\in P'_5.$$
Indeed, the equality $A_{1}(x)=A_{2}(y)=A_{1}(y)+A_{2}(x)=0$ shows that 
$$A_{3}(x+y)=(A_1+A_2)(x+y)=A_1(x)+A_{2}(x)+A_{1}(y)+A_{2}(y)=0.$$
and
$$A_{4}(\alpha^2x+y)=(\alpha^2 A_1+A_2)(\alpha^2 x+y)=\alpha^4A_1(x)+\alpha^2A_{2}(x)+\alpha^2A_{1}(y)+A_{2}(y)=0.$$
Similarly, we establish that $A_{5}(\alpha x+y)=0$.

Suppose that $n(L')\ge 2$, i.e. there are two distinct $i,j$ such that $P'_i$ and $P'_j$ are contained in $H_i\cap C_i$ and $H_j\cap C_j$,
respectively.
Without loss of generality, we can assume that these indices are $1$ and $2$ 
(for any distinct $i,j$ there is a monomial linear automorphism preserving $L$ and transferring $P_{i},P_{j}$ to $P_{1},P_2$).
Then $x_{1}=y_2=0$ and we rewrite
the condition $A_{1}(y)=A_{2}(x)$ as follows
$$y_3 +y_4 + y_5 = x_3 +\alpha^2 x_{4} +\alpha x_5$$
or 
$$\underbrace{y_3 +x_{3}}_{a}+\underbrace{y_4 + \alpha^2 x_4}_{b}+ \underbrace{y_5 + \alpha x_5}_{c}=0.$$
So, we have $a +b +c=0$,
where 
\begin{enumerate} 
\item[$\bullet$] $a$ is the $3$-rd coordinate of $x+y\in P'_3$,
\item[$\bullet$] $b$ is the $4$-th coordinate of $\alpha^2 x +y \in P'_4$,
\item[$\bullet$] $c$ is the $5$-th coordinate of $\alpha x +y\in P'_5$.
\end{enumerate}
If all these numbers are zero, then each $P'_j$ is contained in $H_j\cap C_j$ and  $L'$ coincides with a certain $L_i$ which contradicts our assumption.  
If two of these numbers are zero, then the remaining number also is zero and we come to the previous case. 
If precisely one of these numbers is zero, then $n(L')=3$.
To complete the proof we need to show that the case when each of these numbers is non-zero, i.e. $n(L')=2$, is impossible.

If $a,b,c$ are mutually distinct non-zero-elements of the field,  then one of the following possibilities is realized:
\begin{enumerate}
\item[(A)] $a=y_3+x_3,$ $b=\alpha a=y_4+\alpha^2x_4$, $c=\alpha^2 a=y_5+\alpha x_5$
which implies that
$$a=y_3+x_3=\alpha^2 y_4+\alpha x_4=\alpha y_5+\alpha^2 x_5.$$
\item[(B)] $a=y_3+x_3$, $b=\alpha^2 a=y_4+\alpha^2x_4$, $c=\alpha a=y_5+\alpha x_5$ which implies that
$$a=y_3+x_3=\alpha y_4+x_4=\alpha^2 y_5+x_5.$$
\end{enumerate}
Each of the vectors $y+x, y+\alpha^2 x, y+\alpha x$ has precisely one zero coordinate (as a simplex vector).
Since the first and second coordinates of these vectors are non-zero,
we have the following possibilities:
\begin{enumerate}
\item[$\bullet$] $4$-th or $5$-th  coordinate of $y+x$ is $0$, 
\item[$\bullet$]  $3$-rd or $5$-th coordinate of $y+\alpha^2 x$ is $0$,
\item[$\bullet$]  $3$-rd or $4$-th coordinate of $y+\alpha x$ is $0$. 
\end{enumerate}
We come to the following cases:
\begin{enumerate}
\item[(I)] $4$-th coordinate of  $y+x$, $5$-th coordinate of $y+\alpha^2 x$ and $3$-rd coordinate of $y+\alpha x$ are zero, i.e.
$$y_4=x_4,\;\;\; y_5=\alpha^2 x_5, \;\;\; y_3=\alpha x_3.$$
\item[(II)] $5$-th coordinate of  $y+x$, $3$-rd coordinate of $y+\alpha^2 x$ and $4$-th coordinate of $y+\alpha x$ are zero, i.e.
$$y_5=x_5,\;\;\; y_3=\alpha^2 x_3, \;\;\; y_4=\alpha x_4.$$
\end{enumerate}
Combining (A) and (I) we obtain that 
$$x_3=\alpha a,\;\;\;x_4=a,\;\;\;x_5=\alpha^2 a.$$
Then $A_{1}(x)=0$ implies that 
$$x_2=x_3+x_4+x_5=\alpha a+a+\alpha^2 a=0$$
which contradict that $x$ is a simplex vector
(recall that $x_1=0$ by our assumption).
Similarly, (A) and (II) show that $x_1=x_2=0$ and we get a contradiction again.

Using (B) and (I) we establish that
$$y_3=\alpha^2 a,\;\;\;y_4=\alpha a,\;\;\;y_5= a.$$
Then $A_{2}(y)=0$ implies that
$$y_1=y_3+\alpha^2 y_4+\alpha y_5=\alpha^2 a + a + \alpha a =0$$
and the vector $y$ is not simplex (since $y_2=0$ by our assumption).
Similarly, (B) and (II) show that $y_1=y_2=0$ and we come to a contradiction.
\end{proof} 

\subsection{Action of the group ${\rm G}(L)$}
In this subsection, we describe the action of the group ${\rm G}(L)$ on the set of all simplex lines 
and complete the proof of Theorems \ref{theorem1} and \ref{theorem2}. 
Recall that ${\rm G}(L)$ is the group of all projective transformations induced by monomial semilinear automorphisms of $V$ preserving $L$.
This group is isomorphic to ${\rm P\Gamma L}(2, 4)$ and the permutation group $S_5$, i.e. is of order $120$.
The subgroup of ${\rm G}(L)$ formed by transformations induced by linear automorphisms is isomorphic to ${\rm PGL}(2,4)$
and the alternating  group $A_5$, i.e. is of order $60$.
Therefore, every permutation of the points of $L$ can be uniquely extended to an element of ${\rm G}(L)$;
an element of ${\rm G}(L)$ is induced by a linear automorphism if and only if it gives an even permutation of the points of $L$.

\begin{prop}\label{prop1-1}
The set $\{L_1,\dots,L_6\}$ is an orbit of the action of ${\rm G}(L)$.
\end{prop}

\begin{proof}
Proposition \ref{prop1-conn} shows that every element of ${\rm G}(L)$ preserves $\{L_1,\dots,L_6\}$.
It was established in the previous subsection that for any distinct $i,j\in \{1,\dots,6\}$
there is an element of ${\rm G}(L)$ transferring $L_{i}$ to $L_j$. 
\end{proof}

In the proof of Lemma \ref{lemma-q6-2}, we have constructed the simplex line formed by the points
$$
\langle x \rangle=\langle 0,1,1,\alpha,\alpha\rangle,\;\;
\langle y \rangle=\langle1,0,1,\alpha^2,1\rangle,\;\;
\langle x+y \rangle=\langle 1,1,0,1,\alpha^2\rangle,$$
$$\langle \alpha x+ y \rangle=\langle 1,\alpha,\alpha^2,0,\alpha\rangle,\;\;
\langle \alpha^2 x +y \rangle=\langle 1,\alpha^2,\alpha,\alpha,0\rangle
$$
and such that 
$$x\in H_1\cap C_1,\;y\in H_2\cap C_2,\;x+y\in H_3\cap C_3,\;\alpha x+ y\not\in H_4\cap C_4,\;\alpha^2 x +y\not\in H_5\cap C_5.$$
This line intersects $L_{1}$, $L_3$ and $L_6$ in the points $\langle x\rangle$, $\langle y \rangle$ and $\langle x+y\rangle$,
respectively, and does not intersect other $L_i$ 
(see the matrix presentations of $L_{1},\dots,L_{6}$ after the proof of Lemma \ref{lemma-q6-2}). 
For this reason we denote this line by $L_{136}$.
It was noted above that  $n(L_{136})=3$.
The line $L_{136}$ is identified with the matrix
$$\left[\begin{array}{c@{\;}c@{\;}c@{\;}c@{\;}c}
0&1&1&\alpha&\alpha\\ 
1&0&1&\alpha^2&1\\
1&1&0&1&\alpha^2\\
1&\alpha&\alpha^2&0&\alpha\\ 
1&\alpha^2&\alpha&\alpha&0\\ 
\end{array}
\right]$$
(whose rows correspond to non-zero vectors belonging to the intersections of $L_{136}$ with $C_i$).

\begin{lemma}\label{lemma-G}
${\rm G}(L_{136})\cap {\rm G}(L)$ is a group of order $6$.
\end{lemma}

\begin{proof}
The group ${\rm G}={\rm G}(L_{136})\cap {\rm G}(L)$ preserves the set $\{\langle x\rangle, \langle y\rangle, \langle x+y\rangle\}$ and 
the set formed by the remaining two points of the line $L_{136}$.
Denote by ${\rm G}'$ the subgroup of ${\rm G}(L_{136})$ formed by all elements preserving the above sets. 
The order of ${\rm G}'$ is $12$ (since every permutation of the points of $L_{136}$ can be uniquely extended to an element of ${\rm G}(L_{136})$)
and ${\rm G}\subset {\rm G}'$.
A direct verification shows that  the projective transformations induced by the linear automorphisms
$$d(1,1,1,\alpha,\alpha^2)p_{(1,2,3)}\;\mbox{ and }\;p_{(2,3)(4,5)}$$
belong to ${\rm G}$.
These transformations generate a subgroup ${\rm G}''\subset {\rm G}$ of order $6$
(because $(1,2,3)$ and (2,3) generate the group $S_3$).
All elements from this subgroup induce even permutations of the points of $L_{136}$
(since they are induced by linear automorphisms).  
Then the remaining $6$ elements of ${\rm G}'$ induce odd permutations.
Since ${\rm G}''\subset {\rm G}\subset {\rm G}'$, the group ${\rm G}$ coincides with ${\rm G}'$ or ${\rm G}''$.

Consider  the semilinear automorphism $u$ associated to the non-identity field automorphism 
and such that
$$u(e_{i})=e_{i}\;\mbox{ if }\; i\le 3\;\mbox{ and }\; u(e_4)=\alpha^2e_5,\;\;u(e_5)=\alpha^2 e_4;$$
it transfers every vector $(x_1,x_2,x_3,x_4,x_5)$ to the vector
$$(x^2_1,\,x^2_2,\,x^2_3,\,\alpha^2x^2_5,\,\alpha^2x^2_4).$$
The corresponding projective transformation is an element of ${\rm G}'\setminus {\rm G}''$ and does not belong to ${\rm G}(L)$.
Therefore, ${\rm G}={\rm G}'$ and we get the claim. 
\end{proof}

Since ${\rm G}(L)$ and ${\rm G}(L_{136})\cap {\rm G}(L)$ are of order $120$ and $6$ (respectively),
the orbit containing $L_{136}$ consists of $20$ elements. We denote this orbit by ${\mathcal X}^{3}_{20}$.
Every simplex line $L'$ belonging to the orbit  ${\mathcal X}^{3}_{20}$ intersects precisely $3$ of the lines $L_i$, i.e. $n(L')=3$.

\begin{prop}\label{prop1-2}
A simplex line $L'$ belongs to ${\mathcal X}^{3}_{20}$ if and only if $n(L')=3$.
\end{prop}

\begin{proof}
For any mutually distinct $L_i,L_j,L_k$ there is at most one simplex line inter\-secting all these lines.
Indeed, if there are two such lines, then the subspace spanned by them is at most a hyperplane and \item
contains $L_i,L_j,L_k$ which is impossible by Lemma \ref{lemma-Sij}.
Therefore, the statement follows from the fact that there are precisely $20$ triples of mutually distinct $L_i,L_j,L_k$
and ${\mathcal X}^{3}_{20}$ consists of $20$ elements.
\end{proof}

For any triple of natural numbers $i,j,k$ such that $1\le i<j<k\le 6$ the simplex line intersecting the lines $L_i,L_j,L_k$ will  be denoted by $L_{ijk}$.
See Appendix for the list of all such lines.

\begin{prop}\label{prop1-4}
The action of ${\rm G}(L)$ on the set $\{L_1,\dots,L_6\}$ is sharply $3$-transitive. 
\end{prop}

\begin{proof}
Let $i,j,k$ be a triple of natural numbers satisfying $1\le i<j<k\le 6$
and let $X$ be the set formed by the three points which are the intersections of the line $L_{ijk}$ with the lines $L_{i},L_{j},L_{k}$.
Since $L_{ijk}$ belongs to the orbit  ${\mathcal X}^{3}_{20}$, 
the group ${\rm G}(L_{ijk})\cap {\rm G}(L)$ is similar to the group ${\rm G}(L_{136})\cap {\rm G}(L)$
considered in Lemma \ref{lemma-G}; in particular, every permutation on the set $X$
can be uniquely extended to an element of ${\rm G}(L_{ijk})\cap {\rm G}(L)$.
Therefore, every permutation on the set $\{L_{i},L_{j},L_{k}\}$ can be extended to an element of ${\rm G}(L)$.

For any other triple of natural numbers $i',j',k'$ such that  $1\le i'<j'<k'\le 6$
there is an element of ${\rm G}(L)$ transferring $L_{ijk}$ to $L_{i'j'k'}$.
This transformation sends the set $\{L_i,L_j,L_k\}$ to  the set $\{L_{i'},L_{j'},L_{k'}\}$.

So, the group ${\rm G}(L)$ acts transitively on the set of all ordered triples $L_i,L_j,L_k$ with distinct $i,j,k\in \{1,\dots,6\}$
(we do not require that $i<j<k$). There are precisely $120$ such triples and ${\rm G}(L)$ is of order $120$.
Therefore, the action is sharply transitive.
\end{proof}

Now, we describe the set of all simplex lines $L'$ satisfying $n(L')=0$ and show that they form an orbit of the action of ${\rm G}(L)$. 

Let $w$ be the semilinear automorphism associated to the non-identity field automorphism 
which transfers every vector $(x_1,x_2,x_3,x_4,x_5)$ to the vector
$$(x^2_1,\,x^2_2,\,x^2_3,\,x^2_5,\,x^2_4),$$
i.e. $w$ preserves every $e_{i}$, $i\le 3$ and transposes $e_4,e_5$.
The associated projective transformation belongs to ${\rm G(L)}$ and transposes the lines 
$$
L_{136}=\left[\begin{array}{c@{\;}c@{\;}c@{\;}c@{\;}c}
0&1&1&\alpha&\alpha\\ 
1&0&1&\alpha^2&1\\
1&1&0&1&\alpha^2\\
1&\alpha&\alpha^2&0&\alpha\\ 
1&\alpha^2&\alpha&\alpha&0\\ 
\end{array}
\right],\;\;\;
L_{245}=
\left[\begin{array}{c@{\;}c@{\;}c@{\;}c@{\;}c}
0&1&1&\alpha^2&\alpha^2\\ 
1&0&1&1&\alpha\\ 
1&1&0&\alpha&1\\ 
1&\alpha&\alpha^2&0&\alpha^2\\ 
1&\alpha^2&\alpha&\alpha^2&0\\ 
\end{array}
\right].$$
The group ${\rm G}(L_{136})\cap {\rm G}(L)$ is generated by the projective transformations induced by
$$d(1,1,1,\alpha,\alpha^2)p_{(1,2,3)}\;\mbox{ and }\;p_{(2,3)(4,5)}$$
(see the proof of Lemma \ref{lemma-G}) and a direct verification shows that 
the elements of the group ${\rm G}(L_{136})\cap {\rm G}(L)$ preserve the line $L_{245}$, i.e.
$${\rm G}(L_{136})\cap {\rm G}(L)={\rm G}(L_{245})\cap {\rm G}(L).$$
Denote by ${\rm G}(\{L_{136}, L_{245}\})$ the group formed by all elements of ${\rm G}(L)$ 
preserving the set $\{L_{136}, L_{245}\}$.
This group is generated by ${\rm G}(L_{136})\cap {\rm G}(L)$ and the projective transformation induced by $w$, i.e. is of order $12$.

Observe that $\{1,3,6\}\cap \{2,4,5\}=\emptyset$
and each of the lines $L_{1},\dots, L_{6}$ intersects precisely one of the lines $L_{136}, L_{245}$
in the point corresponding to an $i$-row with $i\le 3$.
The $4$-th point $\langle 1,\alpha,\alpha^2,0,\alpha \rangle$ of the line $L_{136}$ and $5$-th point $\langle 1,\alpha^2,\alpha,\alpha^2,0\rangle$ of the line $L_{245}$
are connected by the simplex line
$$L'=
\left[\begin{array}{c@{\;}c@{\;}c@{\;}c@{\;}c}
0&1&1&\alpha^2&\alpha\\ 
1&0&1&1&1\\ 
1&1&0&\alpha&\alpha^2\\ 
1&\alpha&\alpha^2&0&\alpha\\ 
1&\alpha^2&\alpha&\alpha^2&0\\ 
\end{array}
\right].$$
Similarly,  the $5$-th point $\langle 1,\alpha^2,\alpha,\alpha,0 \rangle$ on $L_{136}$ and the $4$-th point $\langle 1,\alpha,\alpha^2,0,\alpha^2\rangle$ on $L_{245}$
are joined by the simplex line
$$L''=
\left[\begin{array}{c@{\;}c@{\;}c@{\;}c@{\;}c}
0&1&1&\alpha&\alpha^2\\ 
1&0&1&\alpha^2&\alpha\\ 
1&1&0&1&1\\ 
1&\alpha&\alpha^2&0&\alpha^2\\ 
1&\alpha^2&\alpha&\alpha&0\\ 
\end{array}
\right];$$
see the figure below.  
These lines are disjoint with $L$ and all $L_i$ (see the list of the lines $L_i$ after the proof of Lemma \ref{lemma-q6-2}), i.e.
$n(L')=n(L'')=0$.
Note that the projective transformation indued by $p_{(2,3)(4,5)}$ transposes $L',L''$
and these lines belong to the same orbit of the action of ${\rm G}(L)$.
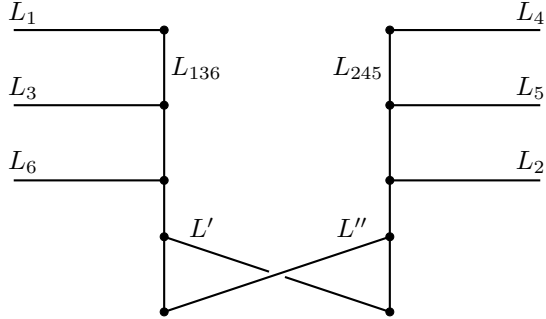
\begin{center}
\begin{tikzpicture}[scale=0.5]

\draw[fill=black] (-3,0) circle (3pt);
\draw[fill=black] (3,0) circle (3pt);

\draw[fill=black] (-3,-2) circle (3pt);
\draw[fill=black] (3,-2) circle (3pt);

\draw[fill=black] (-3,-4) circle (3pt);
\draw[fill=black] (3,-4) circle (3pt);

\draw[fill=black] (-3,-5.5) circle (3pt);
\draw[fill=black] (3,-5.5) circle (3pt);

\draw[fill=black] (-3,-7.5) circle (3pt);
\draw[fill=black] (3,-7.5) circle (3pt);

\draw [thick] (-3,0) -- (-3,-7.5); 
\draw [thick] (3,0) -- (3,-7.5); 

\draw [thick] (-3,0) -- (-7,0); 
\draw [thick] (3,0) -- (7,0); 
\draw [thick] (-3,-2) -- (-7,-2); 
\draw [thick] (3,-2) -- (7,-2); 
\draw [thick] (-3,-4) -- (-7,-4); 
\draw [thick] (3,-4) -- (7,-4); 

\draw [thick] (-3,-5.5) -- (3,-7.5); 
\draw[fill=white, white] (0,-6.5) circle (6pt);
\draw [thick] (3,-5.5) -- (-3,-7.5); 

\node at (-2,-5.25) {$L'$};
\node at (2,-5.25) {$L''$};

\node at (-6.75,0.45) {$L_{1}$};
\node at (-6.75,-1.55) {$L_{3}$};
\node at (-6.75,-3.5) {$L_{6}$};
\node at (6.75,0.45) {$L_{4}$};
\node at (6.75,-1.55) {$L_{5}$};
\node at (6.75,-3.5) {$L_{2}$};

\node at (-2.15,-1) {$L_{136}$};
\node at (2.15,-1) {$L_{245}$};

\end{tikzpicture}
\end{center}

Since  ${\rm G}(L)$ permutes the set of all $L_{i}$,
every element of this group sending $L_{136}$ to $L_{ijk}$ sends also $L_{245}$ to $L_{i'j'k'}$ such that  the indices $i,j,k,i',j',k'$ are mutually distinct;
also, it maps $L'$ and $L''$ to simplex lines  ${\tilde L}'$ and ${\tilde L}''$ satisfying $n({\tilde L}')=n({\tilde L}'')=0$. 
There are precisely $10$ pairs of lines $L_{ijk},L_{i'j'k'}$ whose indices are mutually distinct
and ${\rm G}(L)$ acts transitively on the set of all these pairs
(the latter follows from the fact that ${\rm G}(L)$ and ${\rm G}(\{L_{136}, L_{245}\})$ are of order $120$ and $12$, respectively).
Every such pair $L_{ijk},L_{i'j'k'}$ defines a pair of lines ${\tilde L}',{\tilde L}''$  satisfying $n({\tilde L}')=n({\tilde L}'')=0$
and there is an element of ${\rm G}(L)$ transferring ${\tilde L}'$ to ${\tilde L}''$
(it was noted above that $L'$ and $L''$ belong to the same orbit).

So, there is an orbit of the action of ${\rm G}(L)$ consisting of $20$ simplex lines non-adjacent to  $L, L_1,\dots,L_6$.
We denote this orbit by ${\mathcal X}^{0}_{20}$ (see Appendix for the list of all elements of this orbit). 

\begin{prop}\label{prop1-5}
Every simplex line non-adjacent to  $L, L_1,\dots,L_6$ belongs to the orbit ${\mathcal X}^{0}_{20}$.
\end{prop}

\begin{proof}
We have to show that each simplex line non-adjacent to $L, L_1,\dots,L_6$ belongs to the orbit ${\mathcal X}^{0}_{20}$.  
By Proposition \ref{prop-element}, the line $L_1$ is adjacent to precisely $25$ simplex lines. 
There are $10$ lines $L_{1jk}$ with $j,k\in \{2,\dots,6\}$ and $j<k$.
For the remaining $15$ simplex lines $S$ adjacent to $L_1$ we have $n(S)=1$ by Lemma \ref{lemma-nL}
and each of these lines is not adjacent to other $L_i$.
The same holds for each $L_i$.
Therefore, there are precisely $90=15\cdot 6$ simplex lines $S$ satisfying $n(S)=1$.
The set of all such lines will be denoted by ${\mathcal X}^{1}_{90}$.
The number of all simplex lines is
$$162=1+25+20+20+90+6,$$
where $1$ corresponds to the line $L$, $25$ is the number of simplex lines adjacent to $L$, $|{\mathcal X}^{3}_{20}|=|{\mathcal X}^{0}_{20}|=20$,
$|{\mathcal X}^{1}_{90}|=90$ and $6$ corresponds to the number of  all $L_i$.
This gives the claim.
\end{proof}

The lines from ${\mathcal X}^{0}_{20}$ are mutually disjoint (see the list of these lines given in Appendix).
Therefore, the set $$\{L,L_{1},\dots, L_{6}\}\cup {\mathcal X}^{0}_{20}$$ consists of $27$ mutually disjoint lines.
Since each line contains precisely $5$ points, these lines cover all $5\cdot 27=135$ simplex points
and we get a spread for the set of simplex points.

Theorem \ref{theorem1} and the statement (1) of Theorem \ref{theorem2} are proved (Propositions \ref{prop1-conn} -- \ref{prop1-5}).
To complete the proof of Theorem \ref{theorem2}
we describe the action of ${\rm G}(L)$ on
the set ${\mathcal X}^{1}_{90}$ consisting of all simplex lines $S$ satisfying $n(S)=1$
(this set is introduced in the proof of Proposition \ref{prop1-5}) and
the set ${\mathcal A}$ formed by all simplex lines adjacent to $L$.

\begin{prop}\label{prop1-6}
The set ${\mathcal A}$ is the union of two orbits consisting of $10$ and $15$ simplex lines. 
\end{prop}

\begin{proof}
Consider the simplex lines 
$$T_1=\left[\begin{array}{c@{\;}c@{\;}c@{\;}c@{\;}c}
0&1&1&1&\alpha\\ 
1&0&1&\alpha&1\\ 
1&1&0&\alpha^2&\alpha^2\\ 
1&\alpha&\alpha^2&0&\alpha\\ 
1&\alpha^2&\alpha&1&0\\ 
\end{array}
\right],\;\;\;
T_2=\left[\begin{array}{c@{\;}c@{\;}c@{\;}c@{\;}c}
0&1&1&1&\alpha^2\\ 
1&0&1&\alpha&\alpha\\ 
1&1&0&\alpha^2&1\\ 
1&\alpha&\alpha^2&0&\alpha^2\\ 
1&\alpha^2&\alpha&1&0\\  
\end{array}
\right]$$
joining the point $P_5=\langle 1,\alpha^2,\alpha,1,0\rangle$ on the line $L$  
with the points  $\langle 0,1,1,1,\alpha \rangle$ and $\langle 0,1,1,1,\alpha^2\rangle$ on the lines $L_{345}$ and $L_{126}$, respectively.
The projective transformation induced by the semilinear automorphism sending every vector $(x_1,x_2,x_3,x_4,x_5)$ to the vector
$$(x^2_1,x^2_3,x^2_2,x^2_4,x^2_5)$$
preserves the line $L$ and transposes the lines $T_1$ and $T_2$. 
The remaining three elements of ${\mathcal A}$ passing through  $P_5$ are the following 
$$S_1=\left[\begin{array}{c@{\;}c@{\;}c@{\;}c@{\;}c}
0&1&\alpha&\alpha^2&1\\ 
1&0&\alpha^2&\alpha^2&\alpha^2\\ 
1&\alpha&0&\alpha&1\\ 
1&1&1&0&\alpha\\ 
1&\alpha^2&\alpha&1&0\\ 
\end{array}
\right],\;
S_2=\left[\begin{array}{c@{\;}c@{\;}c@{\;}c@{\;}c}
0&1&\alpha&\alpha^2&\alpha\\ 
1&0&\alpha^2&\alpha^2&1\\ 
1&\alpha&0&\alpha&\alpha\\ 
1&1&1&0&\alpha^2\\ 
1&\alpha^2&\alpha&1&0\\ 
\end{array}
\right],\;
S_3=\left[\begin{array}{c@{\;}c@{\;}c@{\;}c@{\;}c}
0&1&\alpha&\alpha^2&\alpha^2\\ 
1&0&\alpha^2&\alpha^2&\alpha\\ 
1&\alpha&0&\alpha&\alpha^2\\ 
1&1&1&0&1\\ 
1&\alpha^2&\alpha&1&0\\ 
\end{array}
\right].$$
Each of these lines does not intersect any $L_{ijk}$.
The element of ${\rm G}(L)$ associated to the linear automorphism $d(\alpha,1,1,1,1)p_{(2,3,4)}$ induces the permutation $(S_1,S_2,S_3)$
on the set of these lines.
Since the group ${\rm G}(L)$ acts transitively on the points of $L$,
through every point of $L$ there are two simplex lines  intersecting some $L_{ijk}$ and three lines disjoint with all $L_{ijk}$.
The set ${\mathcal A}$ is the union of two orbits:
all elements of ${\mathcal A}$ intersecting  some $L_{ijk}$ form an orbit of size $10$
and all elements of ${\mathcal A}$ disjoint with all  $L_{ijk}$ form an orbit of size $15$.
\end{proof}

\begin{prop}\label{prop1-7}
The set ${\mathcal X}^{1}_{90}$ is the union of two orbits consisting of $30$ and $60$ simplex lines. 
\end{prop}

\begin{proof}
Consider the five simplex lines intersecting $L_1$ in the point $Q_1=\langle 0,1,1,\alpha,\alpha,\rangle$. 
Two of them are the lines $L_{136}, L_{125}$ and 
the remaining three 
$$N_{1}=\left[\begin{array}{c@{\;}c@{\;}c@{\;}c@{\;}c}
0&1&1&\alpha&\alpha\\ 
1&0&1&1&\alpha^2\\ 
1&1&0&\alpha^2&1\\ 
1&\alpha^2&\alpha&0&\alpha\\ 
1&\alpha&\alpha^2&\alpha&0\\ 
\end{array}
\right],\;\;
N_{2}=
\left[\begin{array}{c@{\;}c@{\;}c@{\;}c@{\;}c}
0&1&1&\alpha&\alpha\\ 
1&0&\alpha^2&\alpha^2&\alpha\\ 
1&\alpha^2&0&\alpha&\alpha^2\\ 
1&\alpha&1&0&1\\ 
1&1&\alpha&1&0\\ 
\end{array}
\right],\;\;
N_{3}=
\left[\begin{array}{c@{\;}c@{\;}c@{\;}c@{\;}c}
0&1&1&\alpha&\alpha\\ 
1&0&\alpha&1&\alpha\\
1&\alpha&0&\alpha&1\\
1&\alpha^2&1&0&\alpha^2\\
1&1&\alpha^2&\alpha^2&0\\ 
\end{array}
\right]$$
belong to ${\mathcal X}^{1}_{90}$.
The lines $N_1$ and $N_2$ intersects the lines $L_{236},L_{356}$ and $L_{235},L_{256}$, respectively. 
The line $N_3$ is disjoint with all $L_{ijk}$.
The transformation $f\in {\rm G}(L)$ induced by the semilinear automorphism sending every vector $(x_1,x_2,x_3,x_4,x_5)$ to the vector
$$(x^2_1,x^2_4,x^2_5,x^2_3,x^2_2)$$
preserves the line $L_1$ and transposes the lines $N_1,N_2$.
The transformation $g\in{\rm G}(L)$ induced by the semilinear automorphism sending every vector $(x_1,x_2,x_3,x_4,x_5)$ to the vector
$$(\alpha^2 x^2_5,\alpha^2 x^2_2,x^2_1,\alpha x^2_3,x^2_4)$$
preserves the line $L_1$.
So, $f$ and $g$ belong to ${\rm G}(L)\cap {\rm G}(L_1)$.
Recall that 
$$L_1=\left[\begin{array}{c@{\;}c@{\;}c@{\;}c@{\;}c}
0&1&1&\alpha&\alpha\\
1&0&\alpha&\alpha&1\\
1&\alpha&0&1&\alpha\\
1&1&\alpha^2&0&\alpha^2\\
1&\alpha^2&1&\alpha^2&0\\
\end{array}
\right]$$
and for every $i\in \{2,\dots,5\}$ we denote by $Q'_{i}$
the point on the line $L_1$ with zero $i$-coordinate.
The transformations $f$ and $g$ induce the permutations 
$$(Q'_2,Q'_5,Q'_3, Q'_4)\;\mbox{ and }\; (Q_1,Q'_3,Q'_4,Q'_5)$$
(respectively) on the points of $L_1$.
The composition $fg$ gives the permutation 
$$(Q_1,Q'_4,Q'_3,Q'_2,Q'_5).$$
Therefore, the group ${\rm G}(L)\cap {\rm G}(L_1)$ acts transitively on the points of $L_{1}$. 
Since ${\rm G}(L)$ acts transitively on the set $\{L_1,\dots,L_6\}$, 
the set ${\mathcal X}^{1}_{90}$ is the union of two orbits:
all elements of ${\mathcal X}^{1}_{90}$ intersecting some $L_{ijk}$ form an  orbit of size $60$ and 
the remaining elements of ${\mathcal X}^{1}_{90}$ form an orbit of size $30$.
\end{proof}

\section{Appendix}
The orbit ${\mathcal X}^{3}_{20}$ consists of the following $20$ simplex lines:

$$L_{123}=\left[\begin{array}{c@{\;}c@{\;}c@{\;}c@{\;}c}
\bm{0}&\bm{1}&\bm{\alpha^2}&\bm{\alpha^2}&\bm{\alpha^2}\\ 
\bm{1}&\bm{0}&\bm{\alpha^2}&\bm{1}&\bm{\alpha}\\ 
1&1&0&\alpha&1\\
1&\alpha&\alpha&0&\alpha^2\\
1&\alpha^2&1&\alpha^2&0\\
\end{array}
\right]\!,\,
L_{124}=
\left[\begin{array}{c@{\;}c@{\;}c@{\;}c@{\;}c}
0&1&\alpha&1&\alpha\\
1&0&\alpha&\alpha&1\\
\bm{1}&\bm{1}&\bm{0}&\bm{\alpha^2}&\bm{\alpha^2}\\ 
1&\alpha&1&0&\alpha\\
\bm{1}&\bm{\alpha^2}&\bm{\alpha^2}&\bm{1}&\bm{0}\\ 
\end{array}
\right]\!,\,
L_{125}=
\left[\begin{array}{c@{\;}c@{\;}c@{\;}c@{\;}c}
0&1&1&\alpha&\alpha\\
\bm{1}&\bm{0}&\bm{\alpha^2}&\bm{\alpha}&\bm{\alpha^2}\\ 
\bm{1}&\bm{\alpha^2}&\bm{0}&\bm{\alpha^2}&\bm{\alpha}\\ 
1&1&\alpha&0&1\\
1&\alpha&1&1&0\\
\end{array}
\right]\!,\,$$
$$L_{126}=\left[\begin{array}{c@{\;}c@{\;}c@{\;}c@{\;}c}
\bm{0}&\bm{1}&\bm{1}&\bm{1}&\bm{\alpha^2}\\ 
1&0&\alpha&\alpha^2&\alpha^2\\
1&\alpha&0&1&\alpha\\
1&\alpha^2&1&0&1\\
\bm{1}&\bm{1}&\bm{\alpha^2}&\bm{\alpha}&\bm{0}\\ 
\end{array}
\right]\!,\,
L_{134}=
\left[\begin{array}{c@{\;}c@{\;}c@{\;}c@{\;}c}
0&1&\alpha&\alpha&1\\
\bm{1}&\bm{0}&\bm{1}&\bm{\alpha}&\bm{\alpha}\\ 
1&\alpha^2&0&\alpha^2&1\\ 
1&1&\alpha^2&0&\alpha^2\\
\bm{1}&\bm{\alpha}&\bm{\alpha}&\bm{1}&\bm{0}\\ 
\end{array}
\right]\!,\,
L_{135}=
\left[\begin{array}{c@{\;}c@{\;}c@{\;}c@{\;}c}
\bm{0}&\bm{1}&\bm{1}&\bm{\alpha^2}&\bm{1}\\ 
1&0&\alpha&\alpha&1\\ 
1&\alpha&0&\alpha^2&\alpha^2\\ 
\bm{1}&\bm{\alpha^2}&\bm{1}&\bm{0}&\bm{\alpha}\\ 
1&1&\alpha^2&1&0\\
\end{array}
\right]\!,\,$$
$$L_{136}=\left[\begin{array}{c@{\;}c@{\;}c@{\;}c@{\;}c}
0&1&1&\alpha&\alpha\\ 
1&0&1&\alpha^2&1\\
1&1&0&1&\alpha^2\\
\bm{1}&\bm{\alpha}&\bm{\alpha^2}&\bm{0}&\bm{\alpha}\\ 
\bm{1}&\bm{\alpha^2}&\bm{\alpha}&\bm{\alpha}&\bm{0}\\ 
\end{array}
\right]\!,\,
L_{145}=
\left[\begin{array}{c@{\;}c@{\;}c@{\;}c@{\;}c}
0&1&\alpha^2&1&\alpha^2\\
\bm{1}&\bm{0}&\bm{1}&\bm{\alpha^2}&\bm{\alpha^2}\\ 
1&\alpha&0&1&\alpha\\ 
\bm{1}&\bm{\alpha^2}&\bm{\alpha^2}&\bm{0}&\bm{1}\\ 
1&1&\alpha&\alpha&0\\ 
\end{array}
\right],
L_{146}=
\left[\begin{array}{c@{\;}c@{\;}c@{\;}c@{\;}c}
0&1&\alpha^2&\alpha^2&1\\ 
1&0&\alpha^2&1&\alpha^2\\ 
\bm{1}&\bm{1}&\bm{0}&\bm{\alpha}&\bm{\alpha}\\ 
\bm{1}&\bm{\alpha}&\bm{\alpha}&\bm{0}&\bm{1}\\ 
1&\alpha^2&1&\alpha^2&0\\
\end{array}
\right]\!,\,$$
$$L_{156}=\left[\begin{array}{c@{\;}c@{\;}c@{\;}c@{\;}c}
\bm{0}&\bm{1}&\bm{\alpha}&\bm{1}&\bm{1}\\ 
1&0&1&1&\alpha\\
\bm{1}&\bm{\alpha^2}&\bm{0}&\bm{\alpha}&\bm{1}\\ 
1&1&\alpha^2&0&\alpha^2\\ 
1&\alpha&\alpha&\alpha^2&0\\ 
\end{array}
\right]\!,\,
L_{234}=
\left[\begin{array}{c@{\;}c@{\;}c@{\;}c@{\;}c}
\bm{0}&\bm{1}&\bm{\alpha^2}&\bm{1}&\bm{1}\\ 
1&0&1&\alpha^2&1\\ 
\bm{1}&\bm{\alpha}&\bm{0}&\bm{1}&\bm{\alpha^2}\\ 
1&\alpha^2&\alpha^2&0&\alpha\\
1&1&\alpha&\alpha&0\\ 
\end{array}
\right]\!,\,
L_{235}=
\left[\begin{array}{c@{\;}c@{\;}c@{\;}c@{\;}c}
0&1&\alpha&\alpha&1\\ 
1&0&\alpha&\alpha^2&\alpha^2\\ 
\bm{1}&\bm{1}&\bm{0}&\bm{1}&\bm{\alpha}\\ 
\bm{1}&\bm{\alpha}&\bm{1}&\bm{0}&\bm{1}\\ 
1&\alpha^2&\alpha^2&\alpha&0\\
\end{array}
\right]\!,\,$$
$$L_{236}=\left[\begin{array}{c@{\;}c@{\;}c@{\;}c@{\;}c}
0&1&\alpha&1&\alpha\\ 
\bm{1}&\bm{0}&\bm{1}&\bm{1}&\bm{\alpha^2}\\
1&\alpha^2&0&\alpha&\alpha\\ 
\bm{1}&\bm{1}&\bm{\alpha^2}&\bm{0}&\bm{1}\\ 
1&\alpha&\alpha&\alpha^2&0\\ 
\end{array}
\right]\!,\,
L_{245}=
\left[\begin{array}{c@{\;}c@{\;}c@{\;}c@{\;}c}
0&1&1&\alpha^2&\alpha^2\\ 
1&0&1&1&\alpha\\ 
1&1&0&\alpha&1\\ 
\bm{1}&\bm{\alpha}&\bm{\alpha^2}&\bm{0}&\bm{\alpha^2}\\ 
\bm{1}&\bm{\alpha^2}&\bm{\alpha}&\bm{\alpha^2}&\bm{0}\\ 
\end{array}
\right]\!,\,
L_{246}=
\left[\begin{array}{c@{\;}c@{\;}c@{\;}c@{\;}c}
\bm{0}&\bm{1}&\bm{1}&\bm{\alpha}&\bm{1}\\
1&0&\alpha^2&\alpha&\alpha\\ 
1&\alpha^2&0&\alpha^2&1\\
\bm{1}&\bm{1}&\bm{\alpha}&\bm{0}&\bm{\alpha^2}\\ 
1&\alpha&1&1&0\\
\end{array}
\right]\!,\,$$
$$L_{256}=\left[\begin{array}{c@{\;}c@{\;}c@{\;}c@{\;}c}
0&1&\alpha^2&\alpha^2&1\\ 
\bm{1}&\bm{0}&\bm{1}&\bm{\alpha}&\bm{1}\\ 
1&\alpha&0&\alpha^2&\alpha^2\\ 
1&\alpha^2&\alpha^2&0&\alpha\\ 
\bm{1}&\bm{1}&\bm{\alpha}&\bm{1}&\bm{0}\\ 
\end{array}
\right]\!,\,
L_{345}=
\left[\begin{array}{c@{\;}c@{\;}c@{\;}c@{\;}c}
\bm{0}&\bm{1}&\bm{1}&\bm{1}&\bm{\alpha}\\ 
1&0&\alpha^2&1&\alpha^2\\ 
1&\alpha^2&0&\alpha&\alpha\\ 
1&1&\alpha&0&1\\ 
\bm{1}&\bm{\alpha}&\bm{1}&\bm{\alpha^2}&\bm{0}\\ 
\end{array}
\right]\!,\,
L_{346}=
\left[\begin{array}{c@{\;}c@{\;}c@{\;}c@{\;}c}
0&1&1&\alpha^2&\alpha^2\\ 
\bm{1}&\bm{0}&\bm{\alpha}&\bm{\alpha}&\bm{\alpha^2}\\
\bm{1}&\bm{\alpha}&\bm{0}&\bm{\alpha^2}&\bm{\alpha}\\ 
1&\alpha^2&1&0&1\\
1&1&\alpha^2&1&0\\ 
\end{array}
\right]\!,\,$$
$$L_{356}=\left[\begin{array}{c@{\;}c@{\;}c@{\;}c@{\;}c}
0&1&\alpha^2&1&\alpha^2\\ 
1&0&\alpha^2&\alpha&\alpha\\ 
\bm{1}&\bm{1}&\bm{0}&\bm{\alpha^2}&\bm{1}\\ 
1&\alpha&\alpha&0&\alpha^2\\ 
\bm{1}&\bm{\alpha^2}&\bm{1}&\bm{1}&\bm{0}\\ 
\end{array}
\right]\!,\,
L_{456}=
\left[\begin{array}{c@{\;}c@{\;}c@{\;}c@{\;}c}
\bm{0}&\bm{1}&\bm{\alpha}&\bm{\alpha}&\bm{\alpha}\\ 
\bm{1}&\bm{0}&\bm{\alpha}&\bm{\alpha^2}&\bm{1}\\ 
1&1&0&1&\alpha^2\\ 
1&\alpha&1&0&\alpha\\ 
1&\alpha^2&\alpha^2&\alpha&0\\ 
\end{array}
\right]\!.$$
The two bolded rows correspond to the points which do not belong to the lines $L_1,\dots,L_6$.
We will need them to describe the elements of the orbit ${\mathcal X}^{0}_{20}$.

Consider a pair $L_{ijk}$, $L_{i'j'k'}$ such that $i,j,k,i',k',j'$ are mutually distinct.
There are precisely two $s,t\in \{1,\dots,5\}$ such that the $s$-th and $t$-th points on each of theses lines 
do not belong to any $L_i$, see the table below.
$$\begin{array}{|c|c|c|}
\hline
i,j,k&i',j',k'&s,t\\ \hline
1,3,6&2,4,5&4,5\\ \hline
1,2,4&3,5,6&3,5\\ \hline
1,4,6&2,3,5&3,4\\ \hline
1,3,4&2,5,6&2,5\\ \hline
1,4,5&2,3,6&2,4\\ \hline
1,2,5&3,4,6&2,3\\ \hline
1,2,6&3,4,5&1,5\\ \hline
1,3,5&2,4,6&1,4\\ \hline
1,5,6&2,3,4&1,3\\ \hline
1,2,3&4,5,6&1,2\\ \hline
\end{array}$$
This pair defines two elements of ${\mathcal X}^{0}_{20}$ which will be denoted by $L_{st}$ and $L_{ts}$.
These simplex lines connect the $s$-th and $t$-th points on $L_{ijk}$ with $t$-th and $s$-th points on $L_{i'j'k'}$,
respectively. 
In this way, we get all $20$ elements of the orbit ${\mathcal X}^{0}_{20}$:
$$L_{12}=\left[\begin{array}{c@{\;}c@{\;}c@{\;}c@{\;}c}
0&1&\alpha^2&\alpha^2&\alpha^2\\ 
1&0&\alpha&\alpha^2&1\\ 
1&\alpha^2&0&1&\alpha^2\\
1&1&1&0&\alpha\\
1&\alpha&\alpha^2&\alpha&0\\
\end{array}
\right]\!,\,
L_{21}=
\left[\begin{array}{c@{\;}c@{\;}c@{\;}c@{\;}c}
0&1&\alpha&\alpha&\alpha\\
1&0&\alpha^2&1&\alpha\\
1&\alpha&0&\alpha&1\\ 
1&\alpha^2&\alpha&0&\alpha^2\\
1&1&1&\alpha^2&0\\ 
\end{array}
\right]\!,\,
L_{13}=
\left[\begin{array}{c@{\;}c@{\;}c@{\;}c@{\;}c}
0&1&\alpha&1&1\\
1&0&\alpha^2&\alpha^2&1\\ 
1&\alpha&0&1&\alpha^2\\ 
1&\alpha^2&\alpha&0&\alpha\\
1&1&1&\alpha&0\\
\end{array}
\right]\!,\,$$

$$L_{31}=\left[\begin{array}{c@{\;}c@{\;}c@{\;}c@{\;}c}
0&1&\alpha^2&1&1\\ 
1&0&\alpha&1&\alpha\\
1&\alpha^2&0&\alpha&1\\
1&1&1&0&\alpha^2\\
1&\alpha&\alpha^2&\alpha^2&0\\ 
\end{array}
\right]\!,\,
L_{14}=
\left[\begin{array}{c@{\;}c@{\;}c@{\;}c@{\;}c}
0&1&1&\alpha^2&1\\
1&0&\alpha^2&\alpha^2&\alpha\\ 
1&\alpha^2&0&1&1\\ 
1&1&\alpha&0&\alpha^2\\
1&\alpha&1&\alpha&0\\ 
\end{array}
\right]\!,\,
L_{41}=
\left[\begin{array}{c@{\;}c@{\;}c@{\;}c@{\;}c}
0&1&1&\alpha&1\\ 
1&0&\alpha&1&1\\ 
1&\alpha&0&\alpha&\alpha^2\\ 
1&\alpha^2&1&0&\alpha\\ 
1&1&\alpha^2&\alpha^2&0\\
\end{array}
\right]\!,\,$$

$$L_{15}=\left[\begin{array}{c@{\;}c@{\;}c@{\;}c@{\;}c}
0&1&1&1&\alpha^2\\ 
1&0&\alpha^2&1&1\\
1&\alpha^2&0&\alpha&\alpha^2\\
1&1&\alpha&0&\alpha\\ 
1&\alpha&1&\alpha^2&0\\ 
\end{array}
\right]\!,\,
L_{51}=
\left[\begin{array}{c@{\;}c@{\;}c@{\;}c@{\;}c}
0&1&1&1&\alpha\\
1&0&\alpha&\alpha^2&\alpha\\ 
1&\alpha&0&1&1\\ 
1&\alpha^2&1&0&\alpha^2\\ 
1&1&\alpha^2&\alpha&0\\ 
\end{array}
\right]\!,\,
L_{23}=
\left[\begin{array}{c@{\;}c@{\;}c@{\;}c@{\;}c}
0&1&\alpha&\alpha^2&\alpha^2\\ 
1&0&\alpha^2&\alpha&\alpha^2\\ 
1&\alpha&0&\alpha^2&\alpha\\ 
1&\alpha^2&\alpha&0&1\\ 
1&1&1&1&0\\
\end{array}
\right]\!,\,$$

$$L_{32}=\left[\begin{array}{c@{\;}c@{\;}c@{\;}c@{\;}c}
0&1&\alpha^2&\alpha&\alpha\\ 
1&0&\alpha&\alpha&\alpha^2\\
1&\alpha^2&0&\alpha^2&\alpha\\ 
1&1&1&0&1\\ 
1&\alpha&\alpha^2&1&0\\ 
\end{array}
\right]\!,\,
L_{24}=
\left[\begin{array}{c@{\;}c@{\;}c@{\;}c@{\;}c}
0&1&\alpha&\alpha^2&\alpha\\ 
1&0&1&\alpha^2&\alpha^2\\ 
1&\alpha^2&0&1&\alpha\\ 
1&1&\alpha^2&0&1\\
1&\alpha&\alpha&\alpha&0\\ 
\end{array}
\right]\!,\,
L_{42}=
\left[\begin{array}{c@{\;}c@{\;}c@{\;}c@{\;}c}
0&1&\alpha^2&\alpha&\alpha^2\\ 
1&0&1&1&\alpha^2\\ 
1&\alpha&0&\alpha&\alpha\\ 
1&\alpha^2&\alpha^2&0&1\\ 
1&1&\alpha&\alpha^2&0\\
\end{array}
\right]\!,\,$$

$$L_{25}=\left[\begin{array}{c@{\;}c@{\;}c@{\;}c@{\;}c}
0&1&\alpha^2&\alpha^2&\alpha\\ 
1&0&1&\alpha&\alpha\\ 
1&\alpha&0&\alpha^2&1\\ 
1&\alpha^2&\alpha^2&0&\alpha^2\\ 
1&1&\alpha&1&0\\ 
\end{array}
\right]\!,\,
L_{52}=
\left[\begin{array}{c@{\;}c@{\;}c@{\;}c@{\;}c}
0&1&\alpha&\alpha&\alpha^2\\ 
1&0&1&\alpha&1\\ 
1&\alpha^2&0&\alpha^2&\alpha^2\\ 
1&1&\alpha^2&0&\alpha\\
1&\alpha&\alpha&1&0\\ 
\end{array}
\right]\!,\,
L_{34}=
\left[\begin{array}{c@{\;}c@{\;}c@{\;}c@{\;}c}
0&1&\alpha&\alpha^2&1\\ 
1&0&\alpha&1&\alpha^2\\ 
1&1&0&\alpha&\alpha\\
1&\alpha&1&0&1\\ 
1&\alpha^2&\alpha^2&\alpha^2&0\\
\end{array}
\right]\!,\,$$

$$L_{43}=\left[\begin{array}{c@{\;}c@{\;}c@{\;}c@{\;}c}
0&1&\alpha^2&\alpha&1\\ 
1&0&\alpha^2&\alpha^2&\alpha^2\\ 
1&1&0&1&\alpha\\ 
1&\alpha&\alpha&0&1\\ 
1&\alpha^2&1&\alpha&0\\ 
\end{array}
\right]\!,\,
L_{35}=
\left[\begin{array}{c@{\;}c@{\;}c@{\;}c@{\;}c}
0&1&\alpha^2&1&\alpha\\ 
1&0&\alpha^2&\alpha&1\\ 
1&1&0&\alpha^2&\alpha^2\\ 
1&\alpha&\alpha&0&\alpha\\ 
1&\alpha^2&1&1&0\\ 
\end{array}
\right]\!,\,
L_{53}=
\left[\begin{array}{c@{\;}c@{\;}c@{\;}c@{\;}c}
0&1&\alpha&1&\alpha^2\\ 
1&0&\alpha&\alpha&\alpha\\
1&1&0&\alpha^2&1\\
1&\alpha&1&0&\alpha^2\\
1&\alpha^2&\alpha^2&1&0\\ 
\end{array}
\right]\!,\,$$

$$L_{45}=\left[\begin{array}{c@{\;}c@{\;}c@{\;}c@{\;}c}
0&1&1&\alpha^2&\alpha\\ 
1&0&1&1&1\\ 
1&1&0&\alpha&\alpha^2\\ 
1&\alpha&\alpha^2&0&\alpha\\ 
1&\alpha^2&\alpha&\alpha^2&0\\ 
\end{array}
\right]\!,\,
L_{54}=
\left[\begin{array}{c@{\;}c@{\;}c@{\;}c@{\;}c}
0&1&1&\alpha&\alpha^2\\ 
1&0&1&\alpha^2&\alpha\\ 
1&1&0&1&1\\ 
1&\alpha&\alpha^2&0&\alpha^2\\ 
1&\alpha^2&\alpha&\alpha&0\\ 
\end{array}
\right]\!.$$

\end{document}